\documentclass[12pt,leqno,a4paper]{article}
\usepackage{amsmath}
\usepackage{hyperref}
\usepackage{amssymb}
\usepackage{amsthm}
\usepackage{graphicx,epstopdf}
\usepackage{color}

\makeatletter
\newcommand{\guio}[1]{\nobreakdash-\hspace{0pt}}
\newcommand{\R}{\mathbb{R}}

\newcommand{\T}{\mathbb{T}}

\newtheorem*{theorem*}{Theorem}
\newtheorem{lemma}{Lemma}
\newtheorem*{lemma*}{Lemma}
\newtheorem*{acknowledgements}{Acknowledgements}

\title{The global regularity of vortex patches revisited}

\author{Joan\ Verdera}

\date{}

\begin{document}

\maketitle

\begin{abstract}
We prove persistence of the regularity of the boundary of vortex patches for a  large class of transport equations in the plane. The velocity field is given by convolution of the vorticity with an odd kernel, homogeneous of degree $-1$ and of class $C^2$ off the origin.

\bigskip


\medskip

\end{abstract}

\section{Introduction}\label{intro}

The vorticity form of the Euler equation for incompressible ideal fluids is
\begin{equation}\label{ve}
\begin{aligned}
\partial_{t}\omega+ v\cdot \nabla \omega&=0,\\
v(\cdot,t)&=\omega (\cdot,t)\star \nabla^\perp N,\\
\omega(z,0)&=\omega_{0}(z),
\end{aligned}
\end{equation}
where $\omega(z,t)$ is vorticity at the point $z$ at time $t,$ $z=x+iy \in \mathbb{C}=\mathbb{R}^{2}$, $t\in\mathbb{R},$ and $N(z)=\frac{1}{2\pi} \log|z|$ is the fundamental solution of the laplacian in the plane. Since the kernel in the Biot-Savart law $v(\cdot,t)=\omega (\cdot,t)\star \nabla^\perp N$ is an orthogonal gradient, the velocity field $v$ has zero divergence, in accordance with the fact that the fluid is incompressible.
It is a deep result of Yudovich \cite{Y} that incompressibility yields that \eqref{ve} is well posed in $L^\infty(\R^2)\cap L^1(\R^2),$ in the sense that there exists a unique weak solution of \eqref{ve} for each given initial condition in $L^\infty(\R^2)\cap L^1(\R^2).$

When one takes as initial condition the characteristic function~$\rho_{0}=\chi_{D_{0}}$ of a bounded domain $D_0$, then $\omega(\cdot,t)=\chi_{D_{t}}(\cdot),$ where $D_t$ is a bounded domain. This follows from the fact that \eqref{ve} is a transport equation and one refers to such a solution as a vortex patch. In the eighties the problem of deciding if boundary smoothness persisted for all times was an open challenging question. Chemin   proved in 1993 \cite{Ch} that if  $D_{0}$ is a bounded simply connected domain with boundary of class~$C^{1+\gamma}$, $0<\gamma<1$, then
the weak solution of~\eqref{ve} with initial condition $\chi_{D_0}$ is of the form~$\omega(z,t)=\chi_{D_{t}}(z)$ with $D_{t}$ a simply connected domain of class~$C^{1+\gamma}$ for all times~$t\in\mathbb{R}$.
Indeed in \cite{Ch} one proves a more general statement and the proof depends on para-differential calculus. In \cite{BC} a shorter proof, based on classical analysis methods, was devised. See also \cite{Se}.

In this paper we consider the regularity problem for vortex patches for transport equations of the form

\begin{equation}\label{keq}
\begin{aligned}
\partial_{t}\omega+ v\cdot \nabla \omega&=0,\\
v(\cdot,t)&=\omega (\cdot,t)\star k,\\
\omega(z,0)&=\chi_{D_{0}}(z),
\end{aligned}
\end{equation}
where $k: \mathbb{\R}^2 \setminus \{0 \} \rightarrow \mathbb{\R}^2 $ is an odd kernel, homogeneous of degree $-1,$ of class $C^2$ off the origin, and $D_0$ is a simply connected domain with boundary of class $C^{1+\gamma}$. These kernels produce in general velocity fields $v$ with non-zero divergence. For instance, for $k(z)= \nabla N(z)$ one has $\operatorname{div} v(z,t) = \chi_{D_{t}}(z)$ and the divergence is non-zero but still bounded. This kernel appears in models for aggregation and chemotaxis phenomena in biology (see \cite{BGLV} and the references given there). If $k(z)= \frac{1}{2 \pi z},$ the complex conjugate of $\nabla N(z),$ then
\begin{equation}\label{beu}
\operatorname{div} v(z,t) =  \chi_{D_{t}}(z) \star \operatorname{div} k(z) = \left(-\frac{1}{\pi} \operatorname{p.v.}\frac{x^2-y^2}{|z|^4} \star  \chi_{D_{t}}\right)(z), \quad z \in \mathbb{C},
\end{equation}
which is a second order Riesz transform applied to a characteristic function. The second order Riesz transforms are basic examples of convolution principal value Calder\'on-Zygmund singular integrals, which do not preserve $L^\infty(\mathbb{R}^2).$ Thus, in general, we should expect the divergence to be an unbounded function in $\operatorname{BMO}(\mathbb{R}^2).$ Indeed, if the domain $D_t$ is smoothly bounded, then an even Calder\'on-Zygmund singular integral with kernel of class $C^1$ off the origin sends $\chi_{D_t}$ to a bounded function, an elementary fact that guaranties that $v(z,t)$ is a Lipschitz field. Thus in dealing with equation \eqref{keq} in the  class of smooth vortex patches one has still bounded divergence, but formula \eqref{beu} anticipates serious difficulties.

Our main result reads as follows.

\begin{theorem*}
If $D_{0}$ is a bounded simply connected domain with boundary of class~$C^{1+\gamma}$, $0<\gamma<1$, then
there exists a weak solution of~\eqref{keq} with initial condition $\chi_{D_0}$  of the form~$\omega(z,t)=\chi_{D_{t}}(z)$ with $D_{t}$ a simply connected domain of class~$C^{1+\gamma}$ for all times~$t\in\mathbb{R}$. This solution is unique in the class of characteristic functions of domains of class $C^{1+\gamma}.$
\end{theorem*}

This has been proved simultaneously in any dimension in  {CMOV} via an intricate construction of appropriate defining functions.  Our proof avoids the use of defining functions.

We proceed to state the definition of weak solution to \eqref{keq}, forward in time. A function $\omega(z,t) \in L^{\infty}([0,T), L^1(\mathbb{R}^2) \cap L^{\infty}(\mathbb{R}^2)), \, T>0,$ is a weak solution of \eqref{keq} in $\mathbb{R}^2 \times [0,T)$ if for each $0<T_0<T$ and each compactly supported $\varphi \in C^1(\mathbb{R}^2 \times [0,T_0])$ one has
\begin{equation*}\label{wsol}
\begin{split} \int_{\mathbb{R}^2} \big(\varphi(z,T_0) \omega(z,T_0)-\varphi(z,0) \omega_0(z)\big)\, dz=  \int_0^{T_0} \int_{\mathbb{R}^2}\Big((\partial_t \varphi + v\cdot \nabla \varphi) \, \omega +\varphi \, \omega\, \operatorname{div}(v) \Big)\,dz dt.
\end{split}
\end{equation*}
If $\omega(z,t) \in  C^1\big(\mathbb{R}^2\times [0,T)\big)$ is a classical solution of \eqref{keq}, then it is a weak solution; to show this one just applies the transport formula. Conversely, if $\omega(z,t) \in  C^1\big(\mathbb{R}^2\times [0,T)\big)$ is a weak solution, then it is a classical solution.
See \cite[Chapter 8]{MB} for a discussion of weak solutions for the Euler vorticity equation.

We now briefly outline the proof of the theorem. As in \cite{BC}, first we take for granted that a local in time solution for the CDE (contour dynamics equation \eqref{cde}) exists (see Chapter 8 of \cite{MB}) and then we continue by proving an a priori inequality for the relevant quantities giving the smoothness of $D_t.$ In the appendix we give indications, for the sake of the reader, on how to solve locally in time the CDE in our case (the argument is basically routine). 

The core of the proof is the second step, in which we show the a priori inequality for the 
 H\"older seminorm of order $\gamma$ of the gradient of the local in time solution $X(\cdot,t)$ of the contour dynamics equation. This is achieved by bringing in a commutator,  which appears after an application of Whitney's extension theorem. Is at this moment when we have to leave the boundary of $D_t$ and work in the ambient space.
 
 The third step consists in proving a logarithmic inequality which estimates $\| \nabla v(\cdot,t)\|_\infty$ in terms of $\| \nabla X(\cdot,t) \|_{\gamma, \partial D_0}$ and the Lipschitz constant of the inverse of $X(\cdot,t).$ Again the estimate is performed by resorting to objects that live in the plane but off the boundary of $D_t.$
 
 The hypothesis on the kernel $k$ are the more general one can reasonably think of. The fact that the kernel is odd and homogeneous of degree $-1$  gives that a first order derivative of the kernel is a constant multiple of the delta function plus an even Calder\'on-Zygmund singular integral. The kernels of these even singular integrals are $C^1$ off the origin since $k$ is assumed to be $C^2$ off the origin, and thus send characteristic functions of smoothly bounded domains into bounded functions.

\section{The a priori estimate}\label{sec2}

Parametrize the boundary $\partial D_{0}$ by a mapping $h\colon \mathbb{T}\to \partial D_{0}$,  of class~$C^{1+\gamma}(\mathbb{T},\R^2)$ satisfying the bilipschitz condition
$$
c_{0}^{-1}  \,|u-v| \le |h(u)-h(v)|\le c_{0} \, |u-v|,\quad u,v\in\mathbb{T}
$$
for some positive constant $c_0.$ The mapping $h$ induces a norm in the vector space of $C^{1+\gamma}$ functions on the boundary of $D_0$ with values in the plane defined on
$X \in C^{1+\gamma}(\partial D_0,\R^2)$ by
\begin{equation}\label{norm}
\|X\|_{1+\gamma} = \|X\|_{\infty, \partial D_0} + \|\nabla X\|_{\gamma, \partial D_0},
\end{equation}
where 
$$
\|X\|_{\infty, \partial D_0} = \sup\{|X(\alpha)| : \alpha \in \partial D_0\}
$$
and
$$
\|\nabla X\|_{\gamma, \partial D_0} = \sup \{\frac{|\frac{d}{du}X(h(u))-\frac{d}{du} X(h(v))|}{|u-v|^\gamma}: u, v  \in \mathbb{T}, u\neq v\}.
$$
We have used the notation (which is an abuse of language we adopt for the sake of conciseness) 
$$
\frac{d}{du} f(u)= \frac{d}{d\theta} f(e^{i\theta}), 
$$
for each differentiable function $f$ defined on $\T.$ 
The contour dynamics equation is an ordinary differential equation defined in the open set $\Omega$ of $C^{1+\gamma}(\partial D_0,\R^2)$ consisting of those mappings in  $C^{1+\gamma}(\partial D_0,\R^2)$ satisfying the bilipschitz condition
\begin{equation}\label{bil}
0< b=b(X)=\inf \{\frac{|X(\alpha)-X(\beta)|}{|\alpha-\beta|}: \alpha, \beta \in \partial D_0, \,\alpha\neq\beta\}.
\end{equation}
For more details on the contour dynamics equation (CDE) in this context the reader may consult the last section.

The CDE has a unique solution $X(\cdot,t) \in \Omega$ for $t\in(-T,T)$ satisfying the initial condition $X(\cdot,0)=I,$ where $I$ stands for the identity map on $\partial D_0.$ Then $t \rightarrow X(\alpha,t)$ is the trajectory of the particle that at time $0$ is at the point $\alpha \in \partial D_0$ and
$X(\partial {D_0},t)$ is a Jordan curve of class $C^{1+\gamma}$ which encloses a simply connected domain $D_t.$ The function
$\omega(z,t)=\chi_{D_t}(z)$ is a weak solution of equation \eqref{keq} (forward in time)  in $\mathbb{R}^2 \times [0,T).$ This can be shown as in \cite[proposition 8.6, p.332]{MB}.

Thus the flow associated with the Lipschitz velocity field $v(\cdot,t)=\chi_{D_t} (\cdot,t)\star k,$  namely the solution of 
\begin{equation}\label{flow}
\begin{split}
\frac{d}{dt}X(\alpha,t)&=v(X(\alpha,t),t),\\*[5pt]
X(\alpha,t)&=\alpha\in \R^2,
\end{split}
\end{equation}
coincides with the solution of the CDE for $\alpha \in \partial D_0.$
 What remains to be proven is that the maximal time of existence of the solution of the CDE  is $T=\infty$ and for that we need an a priori inequality. By \eqref{flow}
$$
X(h(u),t)=h(u)+\int_{0}^{t} v(X(h(u),s),s)\,ds, \quad u \in \mathbb{T}.
$$
Hence
\begin{equation}\label{eq2}
\frac{d}{du}X(h(u),t)=\frac{dh}{du}(u)+\int^{t}_{0}\nabla v (X(h(u),s),s)\left(\frac{d}{du}X(h(u),s)\right)\,ds.
\end{equation}

Clearly $\frac{d}{du}X(h(u),s)$  is a tangent vector to~$\partial D_{s}$ at the point $X(h(u),s).$  We would like to show that
$$
\nabla v\big(X(h(u),s),s\big)\left(\frac{d}{du}X(h(u),s)\right)
$$
is a commutator. We need three lemmas. The first works in $\mathbb{R}^n.$  The notion of jet is in \cite[p. 176]{S}, although the term jet is not used there. Given a subset $S$ of $ \subset \mathbb{R}^n$ and a vector field $\vec{F}: S \rightarrow \mathbb{R}^n$ set
$$
\| \vec{F}\|_{\gamma, S}= \sup \{\frac{|\vec{F}(x) -\vec{F}(y)|}{|x-y|^\gamma}: x, y  \in S, x\neq y \}.
$$

\begin{lemma}\label{jet}
	Let $\Gamma$ be a hypersurface of dimension~$n-1$ and of class~$C^{1+\gamma}$ in~$\mathbb{R}^{n}$. Let $N(x)=(N_{1}(x),\dotsc,N_{n}(x))$ be a normal field on~$\Gamma$ of class~$C^{\gamma}(\Gamma)$. Then $(0,N(x))$ is a jet of class~$C^{1+\gamma}(\Gamma)$ with constant~$A\|N\|_{\gamma,\Gamma}$, $A$ a constant depending only on~$\gamma$. In other words,
	$$
	\sup_{x,y\in\Gamma,\,x\ne y}|(y-x)\cdot N(x)|\le A\|N\|_{\gamma,\Gamma} \,|y-x|^{1+\gamma}.
	$$
\end{lemma}

\begin{proof}
	Assume, without loss of generality, that $x=0$ and $N(0)=(0,\dotsc,0,N_{n}(0))$ with $N_{n}(0)>0$. Set  $\delta^{-\gamma}=2\frac{\|N\|_{\gamma,\Gamma}}{|N(0)|}$.
	
	If $|y|\ge\delta$, then
	$$
	|y\cdot N(0)|\le |y| |N(0)|=|y|2\|N\|_{\gamma,\Gamma}\,\delta^{\gamma}\le 2\|N\|_{\gamma,\Gamma}\,|y|^{1+\gamma},
	$$
	as required.

If $y\in B(0,\delta)\cap\Gamma$, then
$$
|N(y)-N(0)|\le \|N\|_{\gamma,\Gamma}\,\delta^{\gamma}=\frac{|N(0)|}{2}.
$$
Hence the tangent hyperplane to $\Gamma$ at~$y$ forms an angle of less than~$30^{\circ}$ with the horizontal hyperplane.~Consequently $B(0,\delta)\cap\Gamma$ is the graph of a function~$y_{n}=\varphi (y')$, $y'=(y_{1},\dotsc,y_{n-1})$, of class~$C^{1+\gamma}$, defined on the projection~$U$ of~$B(0,\delta)\cap\Gamma$ on $\{y\in \mathbb{R}^{n}: y_{n}=0\}$. Note that $U$ is open in~$\mathbb{R}^{n-1}$ and that, for $y\in B(0,\delta)\cap \Gamma$, the segment with extremes~$0$ and $y'$ is contained in~$U$ (by the implicit function theorem).~It is also clear that $|\nabla\varphi (y)|\le 1$, $y'\in U$, because of the inclination of tangent hyperplanes with respect to the horizontal hyperplane. 

If $y\in B(0,\delta)\cap \Gamma$ we have
$$
|y\cdot N(0)|\le | \varphi(y')| |N(0)|\le \left(\sup_{|z'|\le |y|,\,z'\in U}|\nabla \varphi (z')|\right) |y'| |N(0)|\le \|\nabla \varphi\|_{\gamma,\Gamma} |y'|^{1+\gamma} |N(0)|.
$$
Set
$$
M(y')=(-\partial_{1}\varphi (y'),\dotsc,-\partial_{n-1}\varphi (y'),1),\quad y'\in U,
$$
which is a normal vector to~$\Gamma$ at $y=(y',\varphi(y'))$. Hence
$$
M(y')=\frac{N(y)}{N_{n}(y)},\quad y\in B(0,\delta)\cap \Gamma,
$$
and
$$
\|\nabla \varphi\|_{\gamma,U}=\|M\|_{\gamma,U}=\left\|\frac{N(y)}{N_{n}(y)}\right\|_{\gamma, U}.
$$
Take $y,z\in B(0,\delta)\cap\Gamma$. Then
$$
|y-z|=(|y'-z'|^{2}+|\varphi (y')-\varphi(z')|^{2})^{1/2}\le \sqrt{2} |y'-z'|.
$$
One has
\begin{equation}\label{eq5}
\left| \frac{N(y)}{N_{n}(y)} -\frac{N(z)}{N_{n}(z)}\right| \le\frac{1}{|N_{n}(y)|} |N(y)-N(z)|+ |N(z)|\frac{|N_{n}(z)-N_{n}(y)|}{|N_{n}(y)||N_{n}(z)|}.
\end{equation}
By the definition of $\delta$
$$
|N_{n}(y)|\ge |N_{n}(0)|-|N_{n}(y)-N_{n}(0)|\ge |N(0)|-\|N\|_{\gamma,\Gamma}\delta^{\gamma}=\frac{|N(0)|}{2}.
$$
Thus the right hand side of \eqref{eq5} is not greater than
$$
2^{3+\gamma/2}\frac{\|N\|_{\gamma,\Gamma}}{|N(0)|} |y'-z'|^{\gamma}.
$$
Therefore
$$
\frac{|y\cdot N(0)|}{|y|^{1+\gamma}}\le\|\nabla \varphi\|_{\gamma,U} |N(0)|\le 2^{3+\gamma/2}\|N\|_{\gamma,\Gamma}
$$
and the lemma follows with $A=2^{3+\gamma/2}$.
\end{proof}

\begin{lemma}\label{lem2}
Let $\Gamma$ be a Jordan curve of class~$C^{1+\gamma}$ in the plane and $\tau(x)=(\tau_{1}(x),\tau_{2}(x))$, $x\in\Gamma,$  a tangent vector to $\Gamma$ at~$x$ with $\tau \in C^{\gamma}(\Gamma,\mathbb{R}^{2})$.~Then there exists $g(x)=(g_{1}(x),g_{2}(x))\in C^{\gamma} (\mathbb{R}^{2},\mathbb{R}^{2})$ such that $g=\tau$ on~$\Gamma$, $\|g\|_{\gamma,\mathbb{R}^{2}}\le C\|\tau\|_{\gamma,\Gamma}$ for a positive constant $C$ depending only on $\gamma,$  and $\operatorname{div}g=0$ in~$\mathbb{R}^{2}$.
\end{lemma}

\begin{proof}
The field $(-\tau_{2}(x),\tau_{1}(x))$ is normal to~$\Gamma$ at each $x\in\Gamma$. By the previous lemma $(0,-\tau_{2}(x),\tau_{1}(x))$ is a $C^{1+\gamma}$ jet on~$\Gamma$  with constant $C_\gamma \, \|\tau\|_{\gamma, \Gamma}$ and by Whitney's extension theorem there is a function~ $\varphi \in C^{1+\gamma}(\mathbb{R}^{2},\mathbb{R})$ such that
$$
\varphi=0,\quad \partial_{1}\varphi=-\tau_{2},\quad \partial_{2}\varphi=\tau_{1},\quad \text{on $\Gamma$}
$$
and $\|\nabla \varphi\|_{\gamma, \mathbb{R}^2} \le C_0\, C_\gamma\, \|\tau\|_{\gamma,\Gamma},$ with $C_0$ an absolute constant  \cite[Theorem 4, p.177]{S}.
Therefore
$$
\tau(x)=(\partial_{2}\varphi(x),-\partial_{1}\varphi(x)),\quad x\in\Gamma.
$$
Set $g=(\partial_{2}\varphi,-\partial_{1}\varphi),$ so that $g$ has no divergence in the plane and $\|g\|_{\gamma, \mathbb{R}^2} \le C\, \|\tau\|_{\gamma,\Gamma},$ where $C=C_0 \,C_\gamma$ depends only on $\gamma.$
\end{proof}

\begin{lemma}\label{lem3}
Let $\tau(x)=(\tau_{1}(x),\tau_{2}(x))\in C^{\gamma}(\Gamma,\R^2)$ be a tangent field on $\Gamma=\partial D_{s}$. Let $g \in C^{\gamma} (\mathbb{R}^{2},\mathbb{R}^{2})$  be the divergence free global field extending $\tau$ provided by Lemma~\ref{lem2}. Then
\begin{equation}\label{eq3}
\nabla v (x,s)(\tau(x)), \quad x \in \Gamma,
\end{equation}
is the restriction to $\Gamma$ of the commutator 
\begin{equation}\label{commut}
\int_{D_{s}}\nabla k(x-y)(g(x)-g(y))\,dy,\quad x \in \R^2.
\end{equation}
\end{lemma}

\begin{proof}
The first component of \eqref{eq3} is
\begin{equation}\label{eq4}
\partial_{1}v_{1}(x,s)\tau_{1}(x)+\partial_{2}v_{1}(x,s)\tau_{2}(x)= (\partial_{1}k_{1}*\chi)(x)\tau_{1}(x)+(\partial_{2}k_{1}*\chi)(x)\tau_{2}(x)
\end{equation}
with $\chi=\chi_{D_{s}}$. In the sense of distributions ($\operatorname{p.v.}$ stands for principal value)
\begin{alignat*}{2}
\partial_{1}k_{1}&=\operatorname{p.v. }\partial_{1}k_{1}+c_{1}\delta_{0}, &\quad\quad c_{1}&=\int_{|x|=1}k_{1}(x)x_{1}\,d\sigma(x)
\intertext{and}
\partial_{2}k_{1}&=\operatorname{p.v. }\partial_{2}k_{1}+c_{2}\delta_{0}, &\quad\quad c_{2}&=\int_{|x|=1}k_{1}(x)x_{2}\,d\sigma(x).
\end{alignat*}

Let $g$ be the field given by Lemma~\ref{lem2}. Since $(\partial_{1}k_{1}*\chi)g_{1}=(\text{p.v. }\partial_{1}k_{1}*\chi)g_{1}+c_{1}\chi g_{1}$ and $\partial_{1}k_{1}*(\chi g_{1})=\text{p.v. }\partial_{1}k_{1}*(\chi g_{1})+c_{1}\chi g_{1}$, we get
$$
(\partial_{1}k_{1}*\chi)g_{1}=(\text{p.v. }\partial_{1}k_{1}*\chi)g_{1}-\text{p.v. }\partial_{1}k_{1}*(\chi g_{1})+\partial_{1}k_{1}*(\chi g_{1}).
$$
Similarly
$$
(\partial_{2}k_{1}*\chi)g_{2}=(\text{p.v. }\partial_{2}k_{1}*\chi)g_{2}-\text{p.v. }\partial_{2}k_{1}*(\chi g_{2})+\partial_{2}k_{1}*(\chi g_{2}).
$$
Note that
$$
\partial_{i}k_{1}*(\chi g_{i})=k_{1}*\partial_{i}(\chi g_{i})=k_{1}*(g_{i}\partial_{i}\chi+\chi\partial_{i}g_{i})
$$
and so, recalling that $g$ has vanishing divergence,
$$
\partial_{1}k_{1}*(\chi g_{1})+\partial_{2}k_{1}*(\chi g_{2})=k_{1}*(g\cdot \nabla \chi +\chi \operatorname{div}g)=k_{1}*(g\cdot \nabla \chi)= 0
$$
since $g_{|\Gamma}=\tau$ is tangent to $\Gamma$ and $\nabla \chi$ normal to~$\Gamma$ at each point of~$\Gamma$. Thus if $x\in \Gamma=\partial D_{s}$
\begin{equation*}
\begin{split}
(\partial_{1}k_{1}*\chi)(x)\tau_{1}(x)+(\partial_{2}k_{1}*\chi)(x)\tau_{2}(x)&=\int_{D_s}\partial_{1}k_{1}(x-y)(g_{1}(x)-g_{1}(y)\,dy\\
&\quad+\int_{D_s}\partial_{2}k_{1}(x-y)(g_{2}(x)-g_{2}(y))\,dy
\end{split}
\end{equation*}
is  indeed the first component of \eqref{commut}. One has the corresponding formula for the second component of $\nabla v(x,s)(\tau(x))$ and the lemma follows.
\end{proof}

Therefore, by the usual commutator estimate in the H\"older seminorm \cite[Lemma, p.26]{BC} or \cite[Lemma 3.2, p.3799]{BGLV},
$$
\|\nabla v(x,s)(\tau(x))\|_{\gamma,\Gamma}\le C\, \left(1+\|\nabla v(\cdot,s)\|_{\infty}\right)\|g\|_{\gamma,\mathbb{R}^{2}}\le C\, \left(1+ \|\nabla v(\cdot,s)\|_{\infty}\right) \|\tau\|_{\gamma,\Gamma}.
$$

From \eqref{eq2} we see that
$$
\left\|\frac{d}{du}X(h(u),t)\right\|_{\gamma,\mathbb{T}}\le \left\|\frac{d}{du} h(u)\right\|_{\gamma,\mathbb{T}}+C\,\int_{0}^{t} \left(1+\|\nabla v(\cdot,s)\|_{\infty}\right)\left\|\frac{d}{du}X(h(u),s)\right\|_{\gamma,\mathbb{T}}\,ds
$$
and then, by Gronwall's lemma,  we get the a priori inequality
$$
\left\|\frac{d}{du}X(h(u),t)\right\|_{\gamma,\mathbb{T}}\le \left\|\frac{d}{du} h(u)\right\|_{\gamma,\mathbb{T}}
\exp \left(C\,\int_{0}^{t} \left(1+\|\nabla v(\cdot,s)\|_{\infty}\right) \,ds\right).
$$
We also have an a priori inequality for
\begin{equation*}
\begin{split}
b(t)&\equiv \inf_{\alpha\ne \beta}\left| \frac{X(\alpha,t)-X(\beta,t)}{|\alpha-\beta|}\right|=\frac{1}{\sup\limits_{\alpha\ne\beta}
\frac{|\alpha-\beta|}{|X(\alpha,t)-X(\beta,t)|}}\\*[5pt]
&= \frac{1}{\|\nabla X^{-1}(\cdot,t)\|_{\infty}} \ge\exp \left(-\int^{t}_{0}\|\nabla v(\cdot,s)\|_{\infty}\,ds\right).
\end{split}
\end{equation*}
Setting
$$
q(t)=\frac{\bigl\|\frac{d}{du}X(h(u),t)\bigr\|_{\gamma,\mathbb{T}}}{b(t)^{1+\gamma}},  \quad -T< t<T,
$$
we finally obtain
\begin{equation}\label{apri}
q(t)\le C\, \exp \left(C\int^{t}_{0} \left(1+ \|\nabla v(\cdot, s)\|_{\infty} \right) \,ds\right), \quad 0< t<T.
\end{equation}

\section{The logarithmic inequality}\label{sec5}

Let $K\in C^{1} (\mathbb{R}^{2}\backslash \{0\})$ be a real function, homogeneous of degree~$-2$ and even. Set
$$
Tf(x)=\text{p.v. }\int_{\mathbb{R}^{2}} K(x-y) f(y)\,dy,
$$
so that $T$ is a convolution, smooth, homogeneous, even Calder\'on--Zygmund operator. 

Let $X\in C^{1+\gamma}(\partial D_{0},\mathbb{R}^{2})$, $D_{0}$ a domain of class~$C^{1+\gamma}$.~Assume that $X$ is bilipschitz onto the image and let $b$ stand for the inverse of Lipschitz norm of the inverse mapping, as in \eqref{bil}.
Since $X(\partial D_{0})$ is a Jordan curve, the simply connected domain $D$ enclosed by the curve satisfies 
$\partial D=X(\partial D_{0})$.

The maximal singular integral of a function $f$ is
$$
T^{*}f(x)=\sup_{\varepsilon>0}\int_{|y-x|>\varepsilon}K(x-y)f(y)\,dy,\quad x\in\mathbb{R}^{2}.
$$

\begin{lemma}\label{lem4}
We have
\begin{equation*}
\begin{split}
T^{*}(\chi_{D})(x) & \le A\left(1+\log^{+}\left(|D|^{1/2} \,\|\frac{d}{du} [X(h(u))]\|_{\gamma, \T} \, \frac{1}{b^{1+\gamma}}\right)\right)\\*[5pt]
& =  A \left(1+\log^{+}\left(|D|^{1/2} \, q(t)\right)\right), \quad x \in \R^2,
\end{split}
\end{equation*}
where $A$ is a constant depending only on~$\gamma$ and the kernel $K$.
\end{lemma}

\begin{proof}
By an elementary argument  \cite[p.409-410]{MOV} it is enough to prove the inequality for points~$x\in\partial D$. Without loss of generality one can assume that $x=0=X(0)$ and that the unitary normal vectors to~$\partial D_{0}$ and $\partial D$ at~$0$ are $(0,1)$. Define $\delta$ by
$$\delta^{-\gamma}=\frac{1}{b^{1+\gamma}} \|\frac{d}{du} [X(h(u))]\|_{\gamma, \T} .$$
Then for $\varepsilon>0$
\begin{equation}\label{eq6}
\begin{split}
\int_{|y|>\varepsilon} K(y) \chi_{D}(y)\,dy & = \int_{\delta >|y|>\varepsilon} K(y)\chi_{D}(y) \, dy+ \int_{|D|^{1/2}>|y|>\delta} K(y)\chi_{D}(y)\,dy
\\*[5pt] & + \int_{|y|>|D|^{1/2}} K(y)\chi_{D}(y)\,dy.
\end{split}
\end{equation}
The third term is estimated by
$$
\int_{|y|>|D|^{1/2}}\frac{1}{|y|^{2}} \chi_{D}(y)\,dy\le 1
$$
and the second by
$$
2\pi \,\sup_{|y|=1}|K(y)| \log \left( \frac{|D|^{1/2}}{\delta}\right)\le A\log^{+} \left(|D|^{1/2}\, \|\frac{d}{du} [X(h(u))]\|_{\gamma, \T}\, \frac{1}{b^{1+\gamma}}\right).
$$
It remains to estimate the first term in the right hand side of~\eqref{eq6}. The argument that follows is similar to that of \cite{BC} in the proof of Proposition 1, p.23--24. Incidentally, Proposition 1 of \cite{BC} was rediscovered in Lemma 5, p. 1142 of  \cite{MOV2}.

Set
\begin{equation*}
D_{\varepsilon,\delta}=\{y\in D:\varepsilon <|y|<\delta\}, \quad H^{-}=\{y=(y_1,y_2)\in \mathbb{R}^{2}: y_{2}<0\} 
\end{equation*}
\text{ and }
\begin{equation*}
S_{\rho}=\{y\in\mathbb{R}^{2},\,|y|=\rho\},\,\rho>0.
\end{equation*}
Since an even kernel with vanishing integral on the unit circumference  has also vanishing integral on half a circumference, we get
\begin{equation*}
\begin{split}
\int_{D_{\varepsilon,\delta}} K(y)\,dy&= \int_{\varepsilon}^{\delta} \left( \int_{\{\theta\in S^{1}:\rho\theta\in  D\}} K(\theta)\,d\sigma(\theta)\right)\frac{d\rho}{\rho}\\*[5pt]
&=\int_{\varepsilon}^{\delta}\left( \int_{\{\theta\in S^{1}:\rho\theta\in  D\}} K(\theta)\,d\sigma(\theta)-
\int_{S^{1}\cap H^{-}} K(\theta)\,d\sigma (\theta)\right)\frac{d\rho}{\rho}\\*[5pt]
&=\int_{\varepsilon}^{\delta}\left( \int_{\{\theta\in S^{1}:\rho\theta\in (D\backslash H^{-})\cup (H^{-}\backslash D)\}} K(\theta)\,d\sigma(\theta)\right)\frac{d\rho}{\rho}
\end{split}
\end{equation*}
and
$$
\left|\int_{D_{\varepsilon,\delta}}K(y)\,dy\right| \le \sup_{|\theta|=1}|K(\theta)| \int^{\delta}_{0} \sigma\{\theta\in S^{1}:\rho\theta\in (D\backslash H^{-})\cup (H^{-}\backslash D)\}\frac{d\rho}{\rho}.
$$
Since the domains $D\backslash H^{-}$ and $H^{-}\backslash D$ are  ``tangential"
$$
\sigma \{\theta\in S^{1} :\rho\theta \in (D\backslash H^{-})\cup (H^{-}\backslash D)\}\le A
\frac{1}{\rho}\sup \{ |X_{2}(\alpha)|: \alpha\in \partial D_{0}\text{ and }|X(\alpha)|=\rho\}.
$$

Take $|X(\alpha)| =\rho<\delta$.  If $\alpha=h(u)$ and $0=h(u_0),$ then
\begin{equation*}
\begin{split}
|X_{2}(\alpha)|&=|X_{2}(h(u))-X_{2}(h(u_0))|\le \|\frac{d}{du} [X_{2}(h(u))]\|_{\gamma, \T}\,|u-u_0|^{1+\gamma}\\*[5pt]
&\le \|\frac{d}{du} [X(h(u))]\|_{\gamma, \T}\,\frac{1}{b^{1+\gamma}}\,|X(\alpha)|^{1+\gamma}\\*[5pt]
&= \|\frac{d}{du} [X(h(u))]\|_{\gamma, \T}\,\frac{1}{b^{1+\gamma}}\, \rho^{1+\gamma}.
\end{split}
\end{equation*}
Thus
\begin{equation*}
\begin{split}
\left|\int_{D_{\varepsilon,\delta}}K(y)\,dy\right| &\le A\|\frac{d}{du} [X(h(u))]\|_{\gamma, \T} \,\frac{1}{b^{1+\gamma}} \int_{0}^{\delta} \rho^{\gamma-1}\,d\rho \\*[5pt]
&=A\, \|\frac{d}{du} [X(h(u))]\|_{\gamma, \T}\,\frac{1}{b^{1+\gamma}} \frac{\delta^{\gamma}}{\gamma}= A.\qedhere
\end{split}
\end{equation*}
\end{proof}

\section{End of the proof}

Having at our disposition the a priori estimate of section 2 and the logarithmic estimate it is a standard matter to complete the proof.
We remark that $\nabla v(\cdot,t)$ is a $2	\times2$ matrix with entries which are either constant multiples of $\chi_{D_t}$ or even homogeneous smooth Calder\'on -Zygmund operators applied to $\chi_{D_t}.$  Inserting the a priori estimate \eqref{apri} into the logarithmic inequality we obtain
$$
\|\nabla v(\cdot,t)\|_\infty 	\le C+C \, \int_0^t \left(1+ \|\nabla v(\cdot,s)\|_\infty \right)\, ds, \quad t \in (0,T).
$$
The factor $|D|^{1/2}$ in the inequality of Lemma \ref{lem4} causes no trouble because of the standard estimate
$$
\|\nabla X(\cdot,t)\|_\infty \le \exp \int_0^t \|\nabla v(\cdot,s)\|_\infty \,ds, \quad t \in (0,T).
$$
Gronwall's lemma then yields
$$
\|\nabla v(\cdot,t)\|_\infty 	\le C \exp(Ct), \quad t \in (0,T).
$$
Hence, for $t \in (0,T),$
$$
\| \frac{d}{du} [X(h(u),t)]\|_{\gamma, \T}  \le C\, \exp(C \exp(Ct)),
$$
and
$$
b(t) \ge C^{-1} \exp(- C\exp(Ct)) \ge  C^{-1} \exp(- C\exp(C T)).
$$
Consequently  $T=\infty.$

\section{Appendix: local in time existence for the CDE}
Our first remark is that if $k: \mathbb{\R}^2 \setminus \{0 \} \rightarrow \mathbb{\R}^2$ is an odd kernel, homogeneous of degree $-1,$ differentiable off the origin, then
\begin{equation}\label{}
k = \partial_1 (x_1 k ) + \partial_2 (x_2 k), \quad x=(x_1,x_2) \in \R^2\setminus \{0\}.
\end{equation}
This follows straightforwardly from Euler's theorem on homogeneous functions.

Assume that $\omega(z,t)=\chi_{D_t}(z)$ is a weak solution of the general equation \eqref{keq}. The velocity field is
\begin{equation*}\label{}
\begin{split}
v(\cdot,t)=\chi_{D_t}\star k &= \chi_{D_t}\star (\frac{\partial}{\partial x}(x k)+\frac{\partial}{\partial y}(y k))\\*[5pt] &= \frac{\partial}{\partial x} \chi_{D_t}\star (x k)+ \frac{\partial}{\partial y}\chi_{D_t}\star (y  k). \\*[5pt]
\end{split}
\end{equation*}
Since $\nabla \chi_{D_t}= -\vec{n} d\sigma = i dz,$   where $d\sigma$ is the arc-length measure on the curve $\partial D_t$ and $dz=dz_{\partial D_t},$ we have 
$$\frac{\partial}{\partial x} \chi_{D_t}=-n_1 d\sigma =-dy\quad \text{and}\quad \frac{\partial}{\partial y} \chi_{D_t}=-n_2 d\sigma =dx.$$ Setting $z=x+iy$ and $w=u+iv$ we get
\begin{equation}\label{vf}
\begin{split}
v(z,t)& = -\int_{\partial D_t} (x-u) k(z-w) \,dv + \int_{\partial D_t} (y-v) k(z-w) \,du \\*[5pt]
&= \int_{\partial D_t} k(z-w) \langle -i(z-w) \cdot \,dw\rangle,
\end{split}
\end{equation}
where $\langle	\cdot,\cdot \rangle$ denotes scalar product in the plane.

The flow mapping is the solution of the ODE
\begin{equation}\label{flo2}
\begin{split}
\frac{d}{dt}X(\alpha,t)&=v(X(\alpha,t),t),\\*[5pt]
X(\alpha,t)&=\alpha\in \R^2.
\end{split}
\end{equation}
Substituting the expression \eqref{vf} for the velocity field in \eqref{flo2} and making the change of variables $w=X(\beta,t)$ we get
\begin{equation*}\label{flow2}
\begin{split}
\frac{d}{dt}X(\alpha,t)&=v(X(\alpha,t),t)= \int_{\partial D_t} k(X(\alpha,t)-w) \langle -i(X(\alpha,t)-w), \,dw\rangle,\\*[5pt]
&= \int_{\partial D_0} k(X(\alpha,t)-X(\beta,t)) \langle -i(X(\alpha,t)-X(\beta,t)), \,\nabla X(\beta,t) (d \beta)\rangle.
\end{split}
\end{equation*}
In the last line  $\beta$ is a generic point of the curve $\partial D_0$  and $d \beta$ the standard complex differential form associated with the curve. Then
$$
\nabla X(\beta,t) (d \beta) = 	\left(\frac{\partial X_1}{\partial \beta_1}(\beta,t) \, d\beta_1 +\frac{\partial X_1}{\partial \beta_2}(\beta,t)\, d\beta_2,\,
  \frac{\partial X_2}{\partial \beta_1}(\beta,t)\, d\beta_1 +\frac{\partial X_2}{\partial \beta_2}(\beta,t)\, d\beta_2\right).
$$
In fact $\nabla X(\beta,t) $ can be understood more intrinsecally as the differential $DX(\beta,t)$ of the mapping $X(\cdot,t)$ as a 
differentiable mapping from the differentiable curve $\partial D_0$ onto the differentiable curve $\partial D_t.$  This mapping takes a tangent vector to $\partial D_0$ at the point $\beta$ into a tangent vector to $\partial D_t$ at the point $X(\beta,t).$ The point is that this operation involves only first derivatives of $X(\cdot,t)$ on $\partial D_0.$ 

Consider the Banach space $C^{1+\gamma}(\partial D_0,\R^2)$ endowed with the norm \eqref{norm} and the open set $\Omega$
consisting of those mappings in  $C^{1+\gamma}(\partial D_0,\R^2)$ satisfying the bilipschitz condition \eqref{bil}. For each $X \in \Omega$ define
\begin{equation}\label{cdef}
F(X)(\alpha) = \int_{\partial D_0} k(X(\alpha)-X(\beta)) \langle -i(X(\alpha)-X(\beta)), \,\nabla X(\beta) (d \beta)\rangle, \quad \alpha \in \partial D_0.
\end{equation}
The CDE is the autonomous ODE in $\Omega$
\begin{equation}\label{cde}
\frac{dX}{dt}=F(X), \quad X(\cdot,0)=I,
\end{equation}
where $I$ is the identity mapping on $\partial D_0.$ 

Of course one has to check that $F(X)$ belongs to $C^{1+\gamma}(\partial D_0,\R^2)$ for each $X\in \Omega.$  After that one wants to apply the existence and uniqueness theorem of Picard, and for that one needs to check that $F$ is locally Lipschitz in $\Omega.$

All these facts are verified routinely and depend essentially on the fact that the kernel appearing in \eqref{cdef} is of class $C^2(\R^2\setminus\{0\}, \R^2)$ and homogeneous of degree $0.$ The reader may consult \cite[Chapter 8]{MB} for the case of the vorticity equation. For instance, when you take a derivative in $\alpha$ in \eqref{cdef} to prove that $F(X)$ is in $C^{1+\gamma}(\partial D_0,\R^2),$ then you get an expression of the form
\begin{equation}\label{derf}
 \int_{\partial D_0} H(X(\alpha)-X(\beta)) L(\nabla X(\beta), d \beta),
\end{equation}
where $H(\cdot)$ is a kernel of homogeneity $-1$ of class $C^1(\R^2\setminus\{0\},\R^2)$ and $ L(\nabla X(\beta), d \beta)$ a linear expression in $\partial X_j / \partial \beta_k$ and $d\beta_j.$ Hence \eqref{derf} may be understood as a (non-convolution) singular integral on the curve $\partial D_0$ applied to a linear combination of the $\partial X_j / \partial \beta_k,$ which are functions satisfying a H\"older condition of order $\gamma.$ Then the result is a function of $\alpha$ in the same space.

An analogous situation appears in checking that $F$ is locally Lipschitz in $\Omega.$

The reader may also consult  \cite{BGLV}, where full details are provided in a similar context.

\begin{acknowledgements}  The author is grateful to J.B.Garnett for many inspiring conversations on the subject.
The author acknowledges support by 2021-SGR-00071 (Generalitat de Catalunya) and PID2020-112881GB-100.
\end{acknowledgements}

\noindent
{\bf Statements and declarations.} The author has no conflict of interests to declare.\medskip

\noindent
{\bf Data availability statement.} Data sharing not applicable to this article as no datasets were generated or analysed during the current study.

\vspace{0.5cm}
{\small
\begin{tabular}{@{}l}
J.\ Verdera,\\
Departament de Matem\`{a}tiques, Universitat Aut\`{o}noma de Barcelona,\\ 
Centre de Recerca Matem\`atica, Barcelona, Catalonia.\\
{\it E-mail:} {\tt joan.verdera@uab.cat}

\end{tabular}}

\end{document}